\documentclass[10pt,oneside]{amsart}
\usepackage{amsmath}
\usepackage{amsthm}
\usepackage{amsfonts}
\usepackage{amssymb}

\newtheorem{theorem}{Theorem}
\newtheorem{lemma}[theorem]{Lemma}
\newtheorem{prop}[theorem]{Proposition}
\newtheorem{corollary}[theorem]{Corollary}

\theoremstyle{definition}

\theoremstyle{remark}
\newtheorem{remark}[theorem]{Remark}
\newcommand{\DD}{{\mathbb D}}

\newcommand{\EE}{{\mathbb E}}

\newcommand{\GG}{{\mathbb G}}

\newcommand{\rR}{{\mathcal R}}
\newcommand{\CC}{{\mathbb C}}

\newcommand{\TT}{{\mathbb T}}

\newcommand{\la}{\lambda}
\DeclareMathOperator{\Aut}{Aut}

\DeclareMathOperator{\tr}{tr}

\renewcommand{\phi}{\varphi}


\begin{document}

\title[SNP problem vs. weak extremals]{Weak extremals in the symmetrized bidisc}

\address{Laval Univerisity, Quebec, Canada / Jagiellonian University, Krakow, Poland}
%\address{Institute of Mathematics, Faculty of Mathematics and Computer Science, Jagiellonian University,  \L ojasiewicza 6, 30-348 Krak\'ow, Poland}
\author{\L ukasz Kosi\'nski}\email{lukasz.kosinski@gazeta.pl}

%\thanks{The work is partially supported by the Iuventus Plus grant}


\maketitle

\begin{abstract}
The main goal of the paper is to study $m$-extremal mappings in the symmetrized bidisc showing that they are rational and $\GG_2$-inner which, in particular, answers a question posed in \cite{Agl-Lyk-You 2013}.
\end{abstract}

\section{Preliminaries}
\subsection{Introduction}
Over 15 years ago N.~Young and J.~Agler in a sequence of papers (see for example \cite{You} and references contained there) devised a new approach to the spectral Nevanlinna-Pick problem. The crucial role was played by the special domain, the so-called symmetrized bidisc. It is a bounded subdomain of $\CC^2$ denoted by $\GG_2$ and given by the formula
$$
\GG_2=\{(s,p):\ |s-\bar s p| + |p|^2<1\}.
$$
The main idea of Agler and Young was to show that if $x_j$ are cyclic, $j=1,\ldots,m$, then the $2\times 2$ spectral Nevanlinna-Pick problem has a solution if and only if there exists a solution to a proper Pick problem in $\GG_2$.

There is a natural link between the spectral Nevanlinna-Pick problem and so-called weak $m$-extremal mappings defined in our previous paper \cite{Kos-Zwo 2014} (we shall give a definition in the second section), see Remark~\ref{uzupelnic}. In principle, the notion of a weak $m$-extremal mappings is similar to the notion of an $m$-extremal mapping introduced in recent papers of J.~Agler, Z.~Lykova and N.~Young. The main difference, beyond a strong connection of weak $m$-extremal mappings with the SNP problem, lies in the fact that in some sense weak $m$-extremal mappings always do exist (see Subsection~\ref{weak} for details). This is not the case for $m$-extremal mappings - a priori we have no guarantee that considered objects exist.

There is a correspondence between weak $m$-extremals in the symmetrized bidisc and $m$-extremals classical Cartan domain  of the first type (recall here that we have shown in \cite{Kos-Zwo 2014} that the weak $m$-extremals and $m$-extremals coincide in a class of domains containing for example the Cartan domains), i.e. the unit ball in the space of $2\times 2$ complex matrices equipped with the operator norm which we shall denote in the sequel by $\rR_I$. This turns out to be very useful as the Cartan domain has a very nice geometry - it is balanced, homogenous. In particular, Schur's algorithm may be applied to this domain.

Using this correspondence we shall show in Theorem~\ref{thrational} that weak $m$-extremal mappings in the symmetrized bidisc are rational and $\GG_2$-inner (for the definition see Subsection~\ref{additional}) extending the result obtained in \cite{Kos-Zwo 2014}. This result applied to $m$-extremals solves a conjecture posed in \cite{Agl-Lyk-You 2013}. Moreover, our method allows us to estimate the degree of weak $m$-extremals.

%Theorem~\ref{rr} leads the author of the paper to the following question:
%\emph{To solve $n\times n$ spectral Nevanlinna Pi}

\bigskip

Here and throughout the paper $\DD$ denotes the unit disc in the complex plane. $\mathcal O(D,G)$ is the space of holomorphic mappings between a domain $D$ and a set $G$. We shall shortly write $\mathcal O(D)$ for $\mathcal O(D, \CC)$. Moreover, $\mathcal O(\bar D, G)$ denotes the space of holomorphic mappings in a neighborhood of $\bar D$ with values in $G$. For a matrix $a$ let $a^\tau$ be a matrix obtained after a permutation of columns of $a$. The transposition of $a$ is denoted by $a^t$. $\partial_s D$ stands for the Shilov boundary of a bounded domain of $\CC^n$. Finally $f^*(\zeta)$ is the non-tangential limit (provided that it exists) of a bounded holomorphic function $f$ at a point $\zeta\in\TT$.

\subsection{Weak $m$-extremal mappings}\label{weak}

Let $D$ be a domain in $\CC^n$. Take pairwise distinct points $\lambda_1,\ldots , \la_m\in \DD$ and $z_1,\ldots, z_m\in D$, $m\geq 2$. Following \cite{Agl-Lyk-You 2013} we say that the interpolation data
\begin{equation}\lambda_j\mapsto z_j,\quad \DD\to D\end{equation}
is \emph{extremally solvable} if there is an analytic disc $h$ in $D$ such that $h(\la_j)=z_j$ for $j=1,\ldots,m$ and there is no $f\in \mathcal O(\bar \DD, D)$ such that $f(\la_j)=z_j$, $j=1,\ldots, m.$

We shall say that an analytic disc $h:\DD\to D$ is a \emph{weak $m$-extremal with respect to pairwise distinct points $\lambda_j$ in $\DD$, $j=1,\ldots, m$,} if the interpolation data $\lambda_j\mapsto h(\la_j)$ is extremally solvable. Naturally, an analytic disc is called to be a \emph{weak $m$-extremal} (or shortly a \emph{weak extremal}) if it is a weak $m$-extremal with respect to some pairwise distinct points $\la_1,\ldots \la_m\in \DD$.

In \cite{Agl-Lyk-You 2013} the authors introduced a stronger notion of $m$-extremal mappings: an analytic disc $h:\DD\to D$ is called \emph{$m$-extremal} (or shortly \emph{extremal}) if for any pairwise distinct points $\lambda_j$ in $\DD$, $j=1,\ldots, m$, it is a weak $m$-extremal with respect to $\lambda_j$. As mentioned in \cite{Kos-Zwo 2014}, $m$-extremals in this sense usually do not exist (for example there are no $m$-extremals in an annulus in the complex plane). This fact justifies the desire of introducing a weaker definition of $m$-extremal mappings.

On the other hand weak $m$-extremals are more natural due to the following self-evident observation: \emph{the interpolation data $\lambda_j\mapsto z_j$, $\DD\to D$ is extremally solvable if and only if there is a weak $m$-extremal $h$  in $D$ with respect to $\la_j$ such that $h(\la_j)=z_j$.}

\subsection{SNP problem and weak $m$-extremals}
The following simple result shows that weak $m$-extremals are naturally connected with the spectral Nevanlinna-Pick problem:
\begin{prop}\label{propt}
Let $\lambda_1,\ldots, \lambda_m$ be pairwise distinct points in the unit disc and let $(s_1,p_1), \ldots, (s_m, p_m)\in \GG_2$ be distinct. Then the following conditions are equivalent:
\begin{enumerate}
\item[(i)] the Pick problem $\lambda_j\mapsto (s_j, p_j)$ for the symmetrized bidisc is solvable;

\item[(ii)] there is $0<t\leq 1 $ and a weak $m$-extremal mapping $h$ with respect to $t\lambda_1,\ldots, t\lambda_m$ such that $h(t\la_j)= (s_j, p_j)$.
\end{enumerate}
\end{prop}

\begin{proof}
If $f:\DD\to \GG_2$ solves the problem $t \lambda_j \mapsto (s_j,p_j)$, $j=1,\ldots,m$, for some $t\leq 1$, then $\lambda\mapsto f(t \lambda)$ solves $\lambda_j\mapsto (s_j, p_j)$, $\DD\to \GG_2$, $j=1,\ldots,m,$ so (ii) easily implies (i).

To show that (i) implies (ii) define $t$ as a infimum of all $s\leq 1$ such that the problem $s\la_j \mapsto (s_j, p_j)$, $j=1,\ldots, m$, has a solution. Clearly $t>0$, a standard argument implies that $t\la_j \mapsto (s_j,p_j)$, $\DD\to \GG_2$, $j=1,\ldots,m$, is solvable and the fact that it extremally solvable follows immediately from the minimality of $t$.
\end{proof}

Proposition~\ref{propt} remains true for more general domains (for example if $\GG_2$ is replaced with a bounded hyperconvex domain). It particular, it holds for the classical Cartan domains.

The symmetrized bidisc may be also given as an image of the classical Cartan domain of the first type $\rR_I=\{x\in \CC^{2\times 2}:\ ||x||<1\}$ under the mapping $$\pi:\CC^{2\times 2}\ni x\mapsto (\tr x, \det x) \in \CC^2$$ (see Lemma~\ref{lem1} for details).

The natural way to study the sprectrall Nevanlinna-Pick problem is to investigate the class of weak $m$-extremals in the symmetrized bidisc. We do it here showing, in particular, the following result which answers the question posed in \cite{Agl-Lyk-You 2013}:
\begin{theorem}\label{thrational} Any weak $m$-extremal in the symmetrized polydisc is rational and $\GG_2$-inner.
\end{theorem}

\section{Basic ideas and tools}

In this section we recall and introduce basic tools that will be used in the sequel. Some ideas are derived from our recent paper \cite{Kos-Zwo 2014}. We recall them for the convenience of the Reader.

\subsection{Symmetrized bidisc vs. bidisc}\label{symvsb}

Recall that $\GG_2$ may be given as the image of the bidisc $\DD^2$ under the mapping $$p:\CC^2\ni (\la_1, \la_2)\mapsto (\la_1 + \la_2 , \la_1 \la_2)\in \CC^2.$$ Moreover, $p|_{\DD^2}: \DD^2\to \GG_2$ is a proper holomorphic mapping and $\Sigma:=\{(2\la, \la^2):\ \la \in \DD\}$ is its locus set ($\Sigma$ is sometimes called the \emph{royal variety of $\GG_2$}). This, in particular, means that $$p|_{\DD^2\setminus p^{-1}(\Sigma)} : \DD^2\setminus p^{-1}(\Sigma) \to \GG_2 \setminus \Sigma$$ is a double branched holomorphic covering.

Thus, any analytic disc in $\GG_2$ omitting $\Sigma$ may be lifted to an analytic disc in $\DD^2$. Therefore, in principle, it is definitely easier to deal with weak $m$-extremals omitting the royal variety $\Sigma$.

The group of automorphims of the symmetrized bidisc is generated by the mappings
\begin{equation}\label{autg}
\GG_2 \ni p(\la_1, \la_2)\mapsto p(m(\la_1), m(\la_2))\in \GG_2,
\end{equation}
where $m$ is a M\"obius function and
\begin{equation}\label{autgl} \GG_2\ni (s,p) \mapsto (\omega s, \omega^2 p)\in \GG_2,
\end{equation}
where $\omega \in \TT$
(see \cite{Jar-Pfl 2004}).

Recall also that the Shilov boundary of $\GG_2$ is equal to $\{(\la_1 + \la_2, \la_1 \la_2):\ \la_1, \la_2\in \TT\}$.

\subsection{Symmetrized bidisc vs. classical Cartan domain}

Let us define
$$\pi:\CC^{2\times 2}\ni x\mapsto (\tr x, \det x)\in \CC^2.$$ Note that $\pi(\rR_I)=\GG_2$.

Any analytic disc in the symmetrized bidisc may be lifted to an analytic disc in $\bar\rR_I$. This gives the following result which is crucial for our considerations (see \cite{Agl} and \cite{Kos-Zwo 2014}):
\begin{lemma}\label{lem1} Let $f:\DD\to \GG_2$ be an analytic disc. Then there is $\phi\in \mathcal O(\DD, \bar \rR_{II})$ such that $f = \pi\circ \phi^\tau$. Moreover, either
\begin{itemize}
\item[-] $f$ is up to an automorphism of $\GG_2$ of the form $(0,f_2)$, or
\item[-] $\phi$ is an analytic disc in $\rR_{II}$.
\end{itemize}
\end{lemma}

It is self evident that if $f:\DD\to \GG_2$ is a weak $m$-extremal and $\phi$ is an analytic disc in $\rR_{II}$ (respectively in $\rR_I$) such that $f = \pi\circ \phi$, then $\phi$ is $m$-extremal in $\rR_{II}$ (resp. in $\rR_I$). Moreover, $(0,f_2)$ is $m$-extremal in the symmetrized bidisc if and only if $f_2$ is a Blaschke product of degree at most $m-1$. Therefore, the problem of describing weak $m$-extremals in the symmetrized bidisc may be reduced to study of $m$-extremals in the classical Cartan domain of the second type. Here we may apply Schur's algorithm which reduces a problem of investigation of $m$-extremals to describing $2$-extremals i.e. complex geodesics. Formulas for complex geodesics in $\rR_{II}$ were found in \cite{Aba}. For the convenience of the Reader we sketch the idea how to get them. Thanks to the transitivity of $\Aut(\rR_{II})$ it suffices to find formulas for a complex geodesic passing through $0$ and an arbitrary point $a\in \rR_{II}$. Moreover, up to a composition with a linear automorphism we may assume that $a=\left(\begin{array}{cc} a_1 & 0 \\ 0 & a_2 \end{array} \right)$, where $0\leq a_2\leq a_1<1$. Then it is clear that any geodesic passing through $0$ and $a$ is of the form $\la \mapsto \left( \begin{array}{cc} \la & 0 \\ 0 & Z(\la) \end{array} \right)$, where $Z\in \mathcal O(\DD, \DD)$ fixes the origin and $Z(a_1)=a_2$. Therefore, Schur's algorithm gives the following:
\begin{lemma}\label{main} Let $f:\DD\to \rR_{II}$ be an $m$-extremal. Then there is $k\leq m-1$, there are $\Phi_1,\ldots,\Phi_k\in \Aut(\mathcal R_{II})$ and there is $Z\in \mathcal O(\DD,\DD)$ such that $$f(\lambda) = \Phi_1(\lambda(\Phi_2(\ldots \la\Phi_{k}\left( \begin{array}{cc} \lambda & 0 \\ 0 & Z(\lambda)\end{array}\right)))),\quad \la\in \DD.$$
\end{lemma}

\bigskip
The group of automorphisms of $\rR_{I}$ is generated by the mappings \begin{equation}\label{autgg}
\Phi_b(x):=(1-b b^*)^{-\frac 12} (b - x)(1- b^* x)^{-1} (1-b^* b )^{\frac 12},\quad x\in \rR_{I},
\end{equation} where $b\in \rR_{I}$,
and by $$x\mapsto U x V,\quad x\in \rR_{I},$$ where $U$ and $V$ are unitary. Note that $\Phi_b(0)=b$ and $\Phi_b(b)=0$. It is clear that that any automorphism of $\rR_{I}$ is a rational mapping.

An automorphism of $\GG_2$ of the form \eqref{autg} induces the automorphism $\Phi_a$ of the form \eqref{autgg}, were $a$ is scalar, and an automorphism \eqref{autgl} induces the automorphism $x\mapsto U x U^t$, where $U=\tilde \omega I$, $\tilde \omega\in \TT$.

The Shilov boundary of $\GG_2$ may be expressed in terms of the Shilov boundary of $\rR_{II}$ as well.
Recall here that $\partial_s \rR_{II}$ consists of symmetric unitary matrices. Then, one may check that
\begin{equation}\label{gshil}
\partial_s \GG_2 =\{\pi(U):\ \text{$U$ is unitary and both $U$ and $U^\tau$ are symmetric}\}.
\end{equation}

We shall use the facts presented above several times. We shall also need the following simple
\begin{remark}\label{remsom} Suppose that the mapping $\phi$ appearing in Lemma~\ref{lem1} is symmetric and $\phi^\tau$ is symmetric too, i.e. $\phi = \left( \begin{array}{cc} \phi_1 & \phi_2 \\ \phi_2 & \phi_1 \end{array}\right)$. Then $a:=(\phi_1 + \phi_2)$ and $b:= (\phi_1-\phi_2)$, are holomorphic selfmappings of $\DD$ such that $p(a, b) = \phi$. In particular, if $f$ is a weak $m$-extremal than at least one of functions $a$ and $b$ is a Blaschke product of degree at most $m-1$.
\end{remark}

In \cite{Kos-Zwo 2014} we have shown that in $\rR_{I}$ and $\rR_{II}$ the class of weak $m$-extremal mappings coincide with the class of $m$-extremal mappings.

\begin{remark}\label{uzupelnic1} Let $\phi:\DD\to \EE$ be an analytic disc. Lifting it to $\bar \rR_I$ we get an analytic disc $\psi$ such that $\phi = \Pi \circ \psi$. There are Blaschke products (finite or infinite) or unimodular constants $B_1$ and $B_2$ such that $\psi_{12} = B_1  h_{12}$ and $\psi_{21} = B_2  h_{21}$, where holomorphic functions $h_{12}$ and $h_{21}$ do not vanish on $\DD$.

Put $h:=\sqrt{h_{12} h_{21}}$, where the branch of the square root is arbitrarily chosen. Then $\tilde \psi:= \left( \begin{array}{cc} \psi_{11} & B_1 h \\ B_2 h & \psi_{22} \end{array} \right)$ lands in $\bar \rR_I$ and $\phi = \Pi \circ \tilde \psi$.

If additionally $\psi$ lies in $\rR_I$, then $\tilde \psi$ lies in $\rR_I$ too.
\end{remark}

It follows from Schur's algorithm that (weak) $m$-extremals in $\rR_I$ and $\rR_{II}$ are proper. Using this, Remark~\ref{uzupelnic1} and properness of the mapping $\Pi|_{\rR_{II}}: \rR_{II}\to \EE$ (see e.g. \cite{Abo-Whi-You 2007} and \cite{Edi-Kos-Zwo 2013}, \cite{Kos-Zwo 2013} for some information about the geometry of $\EE$), where $\EE$ is the tetrablock and $\Pi(x) = (x_{11}, x_{22}, \det x)$, $x=(x_{ij})\in \CC^{2\times 2}$, we get an analogous result for weak $m$-extremals in the symmetrized bidisc (see also \cite{Kos-Zwo 2014}, Remark~15):
\begin{prop}\label{prop} Any weak $m$-extremal mapping in $\GG_2$ is proper. In particular, if $f:\DD\to \GG_2$ is a weak $m$-extremal in the symmetrized bidisc such that $f=\pi \circ \phi$, where $\phi\in \mathcal O (\DD,\rR_I)$, then $|\phi_{12}^*| = |\phi_{21}^*|$ almost everywhere on $\TT$.
\end{prop}

Remark~\ref{uzupelnic1} has also the following consequence:
\begin{remark}\label{uzupelnic2}
Let $\phi:\DD\to \GG_2$ be an analytic disc and let $\psi:\DD\to \bar \rR_{II}$ be such that $\phi = \pi\circ \psi^{\tau}$ and $\psi_{11} = B_1 h$, $\psi_{22} = B_2 h$ for some Blaschke products (or unimodular constants) $B_1$, $B_2$ and a holomorphic function $h$. Then $\tilde \psi:=\left( \begin{array}{cc} h & \psi_{12} \\ \psi_{21} & B_1 B_2 h \end{array} \right)$ is an analytic disc in $\bar\rR_{II}$ such that $\phi = \pi\circ \tilde \psi$ and $|\tilde \psi_{22}|\leq |\tilde \psi_{11}|$.
\end{remark}

\begin{remark}\label{uzupelnic}
If the problem
\begin{equation}
\label{finn}
\la_j\mapsto (s_j,p_j),\quad \DD\to \GG_2,
\end{equation}
has a solution, then there is a solution of the form $\la\mapsto \psi(t\la)$, where $\psi$ is a weak $m$-extremal for $t\la_j$ and $(s_j, p_j)$ and $0<t\leq 1$ is properly chosen. Therefore we believe that study of weak $m$-extremals is important.

Moreover, thanks to Remark~\ref{uzupelnic2} we may always find a solution $\phi$ of \eqref{finn} such that
\begin{equation}\label{konn}
|\phi_{21}|\leq |\phi_{12}|.
\end{equation}

\end{remark}

%Using this we shall show that $Z$ appearing above is a Blschke product. More precisely, we shall show the following:
%\begin{theorem} Let $f:\DD\to \GG_2$ be a weak $m$-extremal. Then there is a $k\leq m$ and there are $a_1,\ldots a_k\in \mathcal R_{II}$ and a Blaschke product $b$ of degree $k-1$ such that $$f(\lambda) = \pi(\Phi_{a_1}(\lambda(\Phi_{a_2}(\ldots \Phi_{a_k}\left( \begin{array}{cc} \lambda & 0 \\ 0 & b(\lambda)\end{array}\right))))).$$

%In particular all weak $m$-extremals are rational and $\GG_2$ inner. \end{theorem}

\subsection{Other tools and definitions}\label{additional}
Since automorphisms of considered domains are rational it is natural to use the notion of a Nash function. Let us recall its definition.

Let $\Omega$ be a subdomain of $\CC^n$. We say that a holomorphic function $f$ on $\Omega$ is a \emph{Nash function} if there is a non-zero complex polynomial $P:\CC^n \times \CC\to \CC$ such that $P(x, f(x))=0$ for $x\in \Omega$. Similarly, a holomorphic mapping $f$ is called a \emph{Nash mapping} if its every component is a Nash function. We shall need the following, classical result (see~\cite{Two}):
\begin{theorem}
The set of Nash functions on a domain $\Omega$ in $\CC^n$ is a subring of the ring of holomorphic functions on $\Omega$.
\end{theorem}

Finally, let us recall that an analytic disc $f:\DD\to \GG_2$ is said to be \emph{$\GG_2$-inner} if $f^*(\zeta)\in \partial_s \GG_2$, for almost all $\zeta\in \TT$.

\section{Proof}

Let $f:\DD\to \GG_2$ be a weak $m$-extremal. If, up to a composition with an automorphism of the symmetrized bidisc, $f$ is of form $f=(0,f_2)$, then $f_2$ is a Blaschke product of degree $m-1$. Otherwise, by Lemma~\ref{lem1}, there is an $m$ extremal $\psi$ in $\mathcal R_{II}$ such that $f=\pi\circ \psi^\tau$. It follows from Lemma~\ref{main} that $\psi(\lambda) = \Phi_{1}(\lambda(\Phi_{2}(\ldots \la \Phi_{k}\left( \begin{array}{cc} \lambda & 0 \\ 0 & Z(\lambda)\end{array}\right))))$, $\la\in \DD$, for some $k\leq m-1$, $\Phi_1,\ldots, \Phi_k \in \Aut(\rR_{II})$, and a holomorphic mappings $Z:\DD\to \DD$. Our aim is to show that
\begin{lemma}\label{ms} $Z$ is a Blaschke product of degree ad most $m-1$.
\end{lemma}

We start the proof of this fact with the following technical result:
\begin{lemma}\label{lembegin}
Let $h\in \mathcal O(\bar \DD, \DD).$ Then $\phi:\la \mapsto p (\la, h(\la)),$ is not a weak $m$-extremal in $\GG_2$ for any $m$.
\end{lemma}

\begin{proof}
By Rouch\'e's theorem the mappings $\la \mapsto \la$ and $\la \mapsto h(\la)$ have one common zero in $\DD$. Therefore, composing $\phi$ with an automorphism of the symmetrized bidisc and a M\"obius function we may assume that $h(\la)=\la g(\la)$, where $g\in \mathcal O(\DD,\bar \DD)$. We may additionally assume that $g$ is not a Nash function.

We may write
$$
p(\la, h(\la)) = \phi(\la) = \pi \left( \begin{array}{cc} \la \alpha(\la) & \la \beta(\la) \\ \la \beta(\la) & \la \alpha(\la)\end{array} \right) =  \pi \left( \begin{array}{cc} \la \alpha(\la) & \la^2 \beta(\la) \\ \beta(\la) & \la \alpha(\la)\end{array} \right),
$$
$\la\in \DD$,
where $\alpha = \frac{1 + g}{2},$ $\beta =\frac{1-g}{2}$. Note that $\beta(0)\in \DD$ (otherwise $\alpha\equiv 0$, whence $f_1\equiv 0$) and therefore the mapping
$$
\la\mapsto \left( \begin{array}{cc} \la \alpha(\la) & \la^2 \beta(\la) \\ \beta(\la) & \la \alpha(\la)\end{array} \right)
$$
is an $m$-extremal in $\rR_I$.

For $c=\beta(0)\in \DD$, let $\Phi_{\tilde c}$ denote the automorphism of $\rR_I$ given by the formula \eqref{autgg} with $\tilde c=\left(\begin{array}{cc} 0 & 0 \\ c & 0 \end{array}\right)$. Let us compute formula for it
$$\Phi_{\tilde c} \left( \begin{array}{cc} x_{11} & x_{12} \\ x_{21} & x_{22} \end{array} \right)=
\left( \begin{array}{cc} \sqrt{1-|c|^2} \frac{-x_{11}}{ 1- \bar c x_{21}} & \frac{-x_{12} - \bar c \det x}{1 - \bar c x_{21}} \\ \frac{-x_{21} + c}{1 - \bar c x_{21}}& \sqrt{1-|c|^2} \frac{-x_{22}}{ 1- \bar c x_{21}} \end{array} \right),
$$
$x=(x_{ij})\in \rR_I$. Note that
\begin{equation}\label{jedn}\Phi_{\tilde c}\left( \begin{array}{cc} \la x_{11} & \la^2 x_{12} \\  x_{21} & \la x_{22} \end{array} \right)=\left( \begin{array}{cc} \la \Phi_{11}(x)  & \la^2 \Phi_{12}(x) \\ \Phi_{21}(x) & \la \Phi_{22}(x) \end{array} \right),
\end{equation}
$\la\in \DD$, $x\in \rR_I$.

Writing
\begin{equation}\label{gl0}
\Phi_{\tilde c}\left( \begin{array}{cc} \la \alpha(\la) & \la^2\beta(\la) \\ \beta(\la) & \la \alpha(\la) \end{array} \right) = \left( \begin{array}{cc} \la \psi_1(\la) & \la^2 \psi_2(\la) \\  \la \psi_3(\la) & \la \psi_1(\la) \end{array} \right),\quad \la\in \DD,
\end{equation}
we see that
$$\psi : \la \mapsto \left( \begin{array}{cc} \psi_1(\la) & \la\psi_2(\la) \\   \psi_3(\la) & \psi_1(\la) \end{array} \right)
$$
either is $m-1$ extremal in $\rR_I$ or it lies in $\partial \rR_I$. Moreover, the mapping
$$\tilde \psi: \la \mapsto \left( \begin{array}{cc} \psi_1(\la) & \psi_2(\la) \\   \la \psi_3(\la) & \psi_1(\la) \end{array} \right)
$$
lies in $\partial \rR_I$, thanks to the relation \eqref{jedn} and the fact that $\left(\begin{array}{cc} \alpha & \beta \\ \beta &\alpha \end{array}\right)$ lies in $\partial \rR_{II}$. Let us consider two cases

1) Assume first that $\psi$ is an analytic disc in $\rR_I$. Then $a:=\psi_1(0)\in \DD$. Let $\Phi_a$ denote the automorphism of $\rR_I$ given by \eqref{autgg} with $a=a\cdot 1\in \CC^{2\times 2}$. Let us write formula for it:
$$
\Phi_a\left( \begin{array}{cc} x_{11} & x_{12}\\ x_{21} & x_{22} \end{array}\right)=
\left(\begin{array}{cc}
\frac{(a-x_{11})(1 - \bar a x_{22}) - \bar a x_{12}x_{21}}{1-\bar a \tr x + \bar a^2 \det x}& \frac{-x_{12}(1-|a|^2)}{1-\bar a \tr x + \bar a^2 \det x}\\ \frac{-x_{21}(1-|a|^2)}{1-\bar a \tr x + \bar a^2 \det x} & \frac{(a - x_{22})(1 - \bar a x_{11}) - \bar a x_{12}x_{21}}{1-\bar a \tr x + \bar a^2 \det x}
\end{array}
\right),
$$ $x=(x_{ij})\in \rR_I$. Note that
$$\Phi_a\left( \begin{array}{cc} \psi_1(\la) & \la \psi_2(\la) \\ \psi_3(\la) & \psi_1(\la)\end{array}\right)=
\left( \begin{array}{cc} \la\chi_1(\la) & \la \chi_2(\la) \\ \chi_3(\la) & \la \chi_1(\la)\end{array}\right),\quad \la\in \DD,
$$
where $\chi_j\in \mathcal O(\DD)$, is $m-1$ extremal and
$$\la\mapsto \Phi_a\left( \begin{array}{cc} \psi_1(\la) &  \psi_2(\la) \\ \la \psi_3(\la) & \psi_1(\la)\end{array}\right)=
\left( \begin{array}{cc} \la\chi_1(\la) & \chi_2(\la) \\ \la \chi_3(\la) &\la \chi_1(\la)\end{array}\right)
$$
lies in the boundary of $\rR_I$. This, in particular, means that $\chi_2$ is a unimodular constant (put $\la=0$), say $\chi_2 \equiv \omega\in \TT$, whence $\chi_1\equiv 0$. Let us denote $\chi=\chi_3$. This gives
\begin{equation}\label{gl}
\left(\begin{array}{cc} \psi_1(\la) & \la \psi_2(\la) \\ \psi_3(\la) & \psi_1(\la) \end{array}\right)=
\left(\begin{array}{cc} \frac{a - \bar a \omega \la \chi(\la)}{1 - \bar a ^2 \omega \la \chi(\la)}& \frac{-\la \omega (1-|a|^2)}{{1 - \bar a ^2 \omega \la \chi(\la)}} \\ \frac{- \chi(\la) (1- |a|^2)}{1 - \bar a ^2 \omega \la \chi(\la)}&\frac{a - \bar a \omega \la \chi(\la)}{1 - \bar a ^2 \omega \la \chi(\la)} \end{array}\right),\quad \la\in \DD.
\end{equation}
Therefore, making use of \eqref{gl0} and \eqref{gl} we get:
$$\la^2 = \frac{\la^2 \beta(\la)}{\beta(\la)} = \frac{\la^2 \omega (1-|a|^2) - \bar c \la^2 (a^2 - \omega \la \chi(\la))}{\la \chi(\la) (1- |a|^2) + c( 1- \bar a^2 \omega \la \chi(\la))},\quad\la\in \DD.$$ This equality provide us with a contradiction ($\chi$ is not a Nash function, as $g$ is not).

2) Now assume that $\psi$ is an analytic disc in $\partial \rR_I$. Since $\tilde \psi$ lies in $\partial \rR_I$ as well we easily find that $|\psi_2|=|\psi_3|$ on $\DD$. This means that $\psi_3=\omega \psi_2$ for some unimodular $\omega$.

Applying a singular value decomposition theorem we see that there is a unitary matrix $U=(u_{ij})$ and an analytic disc $f\in \mathcal O(\DD, \bar \DD)$ such that $$\psi^\tau(\la) = U \left( \begin{array}{cc} 1 & 0 \\ 0 & f(\la)\end{array} \right)U^t,
\quad \la\in \DD.
$$
In other words
$$
\left( \begin{array}{cc} \psi_1(\la) & \la \psi_2(\la) \\ \psi_3(\la) & \psi_1(\la)\end{array}\right) =
\left( \begin{array}{cc}
u_{11} u_{21} + f(\la)  u_{22} u_{12} & u_{11}^2 + f(\la) u_{12}^2 \\ u_{21}^2 + f(\la) u_{22}^2 &u_{11} u_{21} + f(\la)  u_{22} u_{12}
\end{array}\right).
$$ Thus $\la u_{21}^2 + \la f(\la) u_{22}^2 = \omega u_{11}^2 + \omega f(\la) u_{12}^2$, $\la\in \DD$. Since $f$ is not Nash we immediately get a contradiction.

\end{proof}
%\begin{lemma} Let $f:\DD \to \GG_2$ be a weak $m$-extremal. Then, there is an $m$-extremal $\phi :\DD\to \mathcal R_I$ an $m$-extremal such that $f=\pi\circ f$. Moreover, $f=(f_{ij})$ may be chosen so that $f_{11}=f_{22}$ and \end{lemma}

\begin{proof}[Proof of Lemma~\ref{ms}]
Seeking a contradiction suppose the contrary i.e. $Z$ is not a Blaschke product. Take pairwise distinct $\la_1,\ldots, \la_m$ in the unit disc such that $f$ is a weak $m$-extremal with respect to $\la_1,\ldots, \la_m$. Then $Z$ is not $m$-extremal for data $\lambda_j \mapsto Z(\lambda_j)$, $j=1,\ldots, m$, therefore there is a holomorphic function $Z_1:\bar \DD \to \DD$ such that $Z_1(\lambda_j)=Z(\lambda_j)$, $j=1,\ldots, m.$ Modifying it we may additionally assume that $Z_1$ is not a Nash function.

Define
$$\phi_1(\lambda):= \Phi_{1}(\lambda(\Phi_{2}(\ldots \la \Phi_{k} \left( \begin{array}{cc} \lambda & 0 \\ 0 & Z_1(\lambda) \end{array}\right)))),\quad \la\in \DD.
$$
Since $\pi\circ \phi_1^\tau$ and $f$ coincide for $\lambda_j$, $j=1, \ldots, m$, we find that $\pi \circ \phi_1^\tau$ is a weak $m$-extremal in $\GG_2$. In particular, $|(\phi_1)_{11}|=|(\phi_1)_{22}|$ on $\TT$ by Proposition~\ref{prop}. Take a function $h$ non-vanishing on $\DD$ and Blaschke products or unimodular constants $b_{1}$, $b_2$ such that $(\phi_1)_{11} = b_1 h$ and $(\phi_1)_{22} = b_2 h$. Let $b= b_1 b_2$. Note that the mapping
\begin{equation}\label{phi2}
\phi_2:\lambda \mapsto \left( \begin{array}{cc} h(\lambda) & (\phi_1)_{12}(\lambda) \\ (\phi_1)_{12}(\lambda) & b(\lambda) h(\lambda) \end{array}\right)
\end{equation}
maps $\DD$ into $\rR_{II}$ (otherwise $\phi_2$ would be, up to an automorphism of $\rR_{II}$ of the form $\left(\begin{array}{cc} 1 & 0 \\ 0 & b(\la) \end{array}\right)$ which contradicts the fact that $\phi_2$ is not Nash) and it is $m$-extremal in $\rR_{II}$, as $\pi\circ \phi_2^\tau = \pi\circ \phi_1^\tau$, extending past $\bar \DD$.

Therefore there is $l\leq m-1$ and there are $\Psi_1,\ldots \Psi_l$ automorphisms of $\rR_{II}$ and a holomorphic mapping $T:\DD\to \DD$ such that
\begin{equation}\label{phi2a}
\phi_2(\la) = \Psi_1(\lambda \Psi_2(\ldots(\lambda \Psi_l\left(\begin{array}{cc} \lambda & 0 \\ 0 & T(\lambda) \end{array}\right)))),\quad \la\in \DD.
\end{equation} Note that $T$ extends holomorphically past $\bar \DD$ and that it is not a Nash function.

Properties \eqref{phi2} and \eqref{phi2a} are crucial in the sequel. Let us define
$$
\psi(\lambda, \nu):= \Psi_1(\lambda \Psi_2(\ldots(\lambda \Psi_l\left(\begin{array}{cc} \lambda & 0 \\ 0 & \nu \end{array}\right)))),\quad \la, \nu\in \DD.
$$
Directly from \eqref{phi2} we get $$\psi_{22}(\lambda, T(\la)) = b(\lambda) \psi_{11}(\lambda, T(\lambda)),\quad \la\in \DD$$ and $$\psi_{12}(\la, T(\la)) = \psi_{21}(\la, T(\la)),\quad \la\in \DD.$$ Since $T$ is not Nash we find that $\psi_{22}(\lambda, \nu) = b(\la) \psi_{11}(\la,\nu)$ and $\psi_{12}(\la, \nu) = \psi_{21}(\la, \nu)$, for $\la, \nu\in \DD$.

Thus:
\begin{equation}\label{11} \Psi_1(\lambda \Psi_2(\ldots(\lambda \Psi_l\left(\begin{array}{cc} \lambda & 0 \\ 0 & \nu \end{array}\right))))= \left( \begin{array}{cc} \psi_{11}(\lambda, \nu) & \psi_{12}(\lambda, \nu)\\ \psi_{12}(\lambda, \nu) & b(\lambda) \psi_{11}(\lambda_1, \nu)\end{array}\right)
\end{equation} for any $\lambda, \nu \in \DD$.

If $\la$ and $\nu$ lie in $\TT$, then the left side of \eqref{11} lies in the Shilov boundary of $\rR_{II}$. Therefore $\left( \begin{array}{cc} \psi_{11}(\lambda, \nu) & \psi_{12}(\lambda, \nu)\\ \psi_{12}(\lambda, \nu) & b(\lambda) \psi_{11}(\lambda_1, \nu)\end{array}\right)$ is a unitary matrix for any $\la, \nu\in \TT$. Thus the following equations are satisfied for $\la, \nu \in \TT$:
\begin{align}\label{e1} |\psi_{11}(\la, \nu)|^2 +& |\psi_{12}(\la, \nu)|^2 = 1 \\\label{e2}
\overline \psi_{11} (\la, \nu) \psi_{12}(\la, \nu) +& b(\la) \psi_{11}(\la, \nu) \overline \psi_{12}(\la, \nu) = 0.
\end{align}

Fix $\lambda_0\in \TT$ and let $\sqrt{b(\lambda_0)}$ denote any square root of $b(\lambda_0)$. It follows from equations \eqref{e1} and \eqref{e2} that $$|\sqrt{b(\lambda_0)} \psi_{11}(\lambda_0, \nu) + \psi_{12} (\lambda_0, \nu)|=1 $$ and
$$|\sqrt{b(\lambda_0)} \psi_{11}(\lambda_0, \nu) - \psi_{12} (\lambda_0, \nu)|=1$$ for any $\nu\in \TT$. Thus there are Blaschke products or unimodular constants $B_1$ and $B_2$ such that
\begin{equation}\label{eq11}\sqrt{b(\lambda_0)} \psi_{11}(\lambda_0, \nu) + \psi_{12} (\lambda_0, \nu)=B_1(\nu),\quad \nu \in \DD,
\end{equation}
and
\begin{equation}\label{eq12}\sqrt{b(\lambda_0)} \psi_{11}(\lambda_0, \nu) - \psi_{12} (\lambda_0, \nu)=B_2(\nu),\quad \nu \in \DD.
\end{equation}
Putting it to \eqref{11} we get
\begin{multline}\label{12} \Psi_1(\lambda_0 \Psi_2(\ldots(\lambda_0 \Psi_l\left(\begin{array}{cc} \lambda_0 & 0 \\ 0 & \nu \end{array}\right))))=\\ \left( \begin{array}{cc}\frac{1}{ \sqrt{b(\la_0)}} \frac{B_1(\nu) + B_2(\nu)}{2} & \frac{B_1(\nu) - B_2(\nu)}{2}\\  \frac{B_1(\nu) - B_2(\nu)}{2} & \sqrt{b(\la_0)}\frac{B_1(\nu) + B_2(\nu)}{2}\end{array}\right)
\end{multline}
for $\nu\in \DD$. Clearly the matrix in the left side of $\eqref{12}$ lies in the topological boundary of $\rR_{II}$ for any $\nu\in \DD$, as $\la_0\in \TT$, so its operator norm is equal to $1$. On the other hand for any $\nu$ in the unit disc the norm of the matrix in the right side of \eqref{12} is equal to $\max(|B_1(\nu)|, |B_2(\nu)|)$. In particular, it is less then $1$ if $\nu$ lies in $\DD$ provided that both $B_1$ and $B_2$ are not unimodular constants.

Therefore at least one of $B_i$ is constant. Putting $\alpha_1= \frac{\partial \psi_{11}}{\partial \nu}$ and $\alpha_2= \frac{\partial \psi_{12}}{\partial \nu}$ and differentiating the equalities \eqref{eq11} and \eqref{eq12} we easily get that the equality $b(\lambda_0) \alpha_1^2(\lambda_0, \nu) = \alpha_2^2(\lambda_0, \nu)$ holds for any $\nu\in \DD$. Therefore we have shown that
\begin{equation}
b(\la) \alpha_1^2(\la, \nu) = \alpha_2^2(\la, \nu),\quad \la,\nu\in \DD.
\end{equation}

Note that if $\alpha_j$ vanishes identically for some $j=1,2$, then $\psi_{11}(\la, \nu)$ and $\psi_{12}(\la, \nu)$ will not be depended on $\nu$ which is impossible. Therefore we easily infer that there is a Blaschke product or unimodular constant $\tilde b$ such that $b=\tilde b^2$.

Define $\varphi_3(\lambda) = \left( \begin{array}{cc} \tilde b(\lambda) (\phi_1)_{11}(\lambda) & (\phi_1)_{12}(\lambda) \\ (\phi_1)_{12}(\lambda) & \tilde b(\lambda) (\phi_1)_{11}(\lambda) \end{array}\right),$ $\lambda\in \DD$. Note that $\pi\circ \varphi_3^\tau$ is a weak $m$-extremal in $\GG_2$ (it is equal to $\pi\circ \phi_2^\tau$). Moreover, for any $\lambda\in \TT$ the point $\phi_3(\lambda)$ does not lie in the Shilov boundary of $\rR_I$ as $\phi_2(\lambda)$ does not (this follows from the fact that $T(\DD)\subset \subset \DD$) and $\phi_3$ is not Nash.

Using properties of $\phi_3$ we shall construct an $m$-extremal mapping in $\GG_2$ of the form $\pi\circ \phi^{\tau}$, where $\phi$ is an analytic disc in $\rR_I$ extending past $\bar \DD$ such that $\phi$ and $\phi^\tau$ are symmetric, $\pi\circ \phi^\tau(\bar \DD)$ does not touch the Shilov boundary of $\GG_2$, $\phi$ is not Nash and
\begin{enumerate}
\item[(i)] either $\lambda\mapsto \pi(\phi(\la)^\tau)$ is a weak $m$-extremal in $\GG_2$ omitting $\Sigma$, or
\item[(ii)] $\phi$ is a $2$-extremal in $\rR_{I}$ passing through the origin.
\end{enumerate}

To get such a mapping observe that we may assume that $(\phi_1)_{11}$ does not vanish on $\DD$ (otherwise we may include its zeros to $\tilde b$).

If $\tilde b$ is a unimodular constant, we just take $\phi= \phi_3$.

Otherwise, let $\la_1,\ldots, \la_l$ be zeros of $\tilde b$ counted with the multiplicity.

If $l=1$, then we may of course assume that $b(\la)=\la$, $\la\in \DD$. Composing $\phi_3^\tau$ with automorphism $\Phi_a$ of $\rR_{I}$ of the form \eqref{autgg} with a scalar $a$ (recall that such an automorphism induces an automorphism of $\GG_2$) we may assume additionally that $(\phi_1)_{12}(0)=0$. Then we have two possibilities:
\begin{itemize}
\item $\la\mapsto \frac 1\la \phi_3(\la)$ is an analytic disc in $\rR_I$ and then $\phi(\la) = \frac 1\la \phi_3(\la),$ $\la\in \DD$, satisfies (i), or
\item $\la\mapsto \frac 1\la \phi_3(\la)$ lands in the topological boundary of $\rR_I$ and then $\phi_3$ is $2$-extremal, so $\phi=\phi_3$ satisfies (ii).
\end{itemize}

Note that the case when $l>1$ may be reduced to these two possibilities as well. Let $\Phi_{\la_j}$ denote an automorphisms of $\rR_I$ such that $\Phi_{\la_j}(0) = \la_j$ and $\Phi_{\la_j}(\la_j) =0$, $j=1,\ldots, l.$ Since every $\Phi_{\la_j}$ induces an automorphism of $\GG_2$ (see \eqref{11}) we see that applying $l$-times the procedure described above we will obtain in this way a weak $m$-extremal in $\GG_2$ satisfying (i) or (ii).

Note that the situation (i) is impossible. Actually, otherwise one may lift $\phi$ to an $m$-extremal in the bidisc, call it $(a_1, a_2)$. Losing no generality we may assume that $a_1$ is a Blaschke product of degree at most $m$. Then $a_2(\DD)$ is relatively compact in $\DD$, which implies that the equation $a_1=a_2$ has one solution in $\DD$; a contradiction.

If (ii) holds then $\pi\circ\phi^\tau$ is of the form $$\phi:\la \mapsto U \left( \begin{array}{cc} \la & 0 \\ 0 & T_1 (\la)\end{array} \right) U^t $$ is a weak $m$-extremal in $\GG_2$ (actually it is $m-l$-extremal) such that $\phi$ and $\phi^{\tau}$ are symmetric. Simply computations (remember about the symmetry of $\phi^\tau$) lead to a formula $\pi\circ \phi^\tau (\la)= p(\la , T_1(\la))$, $\la \in \DD$. Using Lemma~\ref{lembegin} one can derive a contradiction.

\end{proof}

\begin{proof}[Proof of Theorem~\ref{thrational}]
Rationality of weak $m$-extremals is a consequence of Lemma~\ref{ms}. The fact that they are inner follows immediately from the description of the Shilov boundaries of the symmetrized bidisc and the classical Cartan domains.
\end{proof}

\begin{remark}\label{deg}
Note that we are able to estimate the degree of any weak $m$-extremal in the symmetrized bidisc. To do it more precisely one may repeat the argument used in the proof of Lemma~\ref{ms} and show that for any weak $m$-extremal $f$ in $\GG_2$ there are $k\leq m-1$, pairwise distinct points $\la_1,\ldots,\la_{k-1}$ in $\DD$, automorphisms $\Phi_1,\ldots \Phi_{k}$ of $\rR_I$ and a Blaschke product $b$ of degree $m-k$ such that
$$
f(\la) = \pi(\Phi_1(m_1(\la)\Phi_2(\ldots m_{k-1}(\la) \Phi_{k} \left(\begin{array}{cc} \la & 0\\ 0 & b(\la)\end{array}
\right)))),\quad \la \in \DD,
$$ where $m_j(\la) = \frac{\la_j - \la}{1 - \bar \la_j \la},$ $\la \in \DD$, $j=1,\ldots, k.$
Note that it is easier to estimate the degree if $m$-extremal mappings in $\GG_2$.
\end{remark}

Observe that estimating the degree of a weak $m$-extremal in $\GG_2$ is simple if it omits the royal variety of $\GG_2$:
\begin{remark}
Let $f:\DD\to \GG_2$ be a weak $m$-extremal in $\GG_2$ omitting $\Sigma$. Then we may lift it of an $m$-extremal in $\DD^2$, i.e. there is $a_1$ - a Blaschke product of degree at most $m-1$ and a function $a_2\in \mathcal O(\DD, \DD)$ such that $f=p(a_1, a_2).$ Note that $a_2$ is a Blaschke product of degree at most $m-1$. Otherwise one can find a function $b:\bar \DD\to \DD$ such that $b$ is not Nash. Then $p(a_1,b)$ is a weak $m$-extremal in $\GG_2$ intersecting $\Sigma$ (Rouch\'e's theorem). This gives a contradiction, as we have shown that all weak $m$-extremals in $\GG_2$ intersecting $\Sigma$ are rational.

\end{remark}

\begin{corollary}\label{cor} Let $f:\DD\to \GG_2$ be a weak $m$-extremal mapping. Then $f_2$ is a Blaschke product.
\end{corollary}

\begin{proof} There is an analytic disc in $\rR_{II}$ such that $f=\pi\circ \phi^{\tau}$. It follows from Lemma~\ref{ms} that $\det \phi(\la)$ is a unimodular constant for any $\la\in \TT$.
\end{proof}

{\bf Question.} Recall the following question stated in \cite{Kos-Zwo 2014}: \emph{are weak $m$-extremals in the symmetrized bidisc $m$-extremals}? If the answer is positive, the next question is very natural: \emph{are extremals in the symmetrized bidisc complex geodesics}? For a definition of a complex geodesic see \cite{Kos-Zwo 2014}. Note that it follows from the results obtained in the paper that for any $m$ there is $k$ such that a weak $m$-extremal in $\GG_2$ is a $k$-complex geodesic (see Corollary~\ref{cor} and Remark~\ref{deg}).

\bigskip

{\bf Acknowledgments.} I would like to thank Zinaida Lykova for reading a draft version of the manuscript and many comments that improved the shape of the paper. I am also grateful to Thomas Ransford and Sylwester Zaj\c{a}c for interesting discussions and bringing my attention to some important papers.

\end{document}